\documentclass[letterpaper, 10 pt, conference]{IEEEtran}

\usepackage[sort,compress]{cite}

\IEEEoverridecommandlockouts

\usepackage{amsmath,amssymb,amsfonts}
\usepackage{graphicx}
\usepackage{mathtools}
\usepackage{bm}
\usepackage{xcolor}
\usepackage{nicefrac}
\usepackage[activate={true,nocompatibility},final,tracking=true,kerning=true,spacing=true,factor=1100,stretch=10,shrink=10]{microtype}
\usepackage{eucal}
\usepackage{algorithm}
\usepackage{algorithmic}
\usepackage{physics}
\usepackage[scr]{rsfso}
\usepackage{lipsum}
\usepackage{bm}
\usepackage{optidef}

\usepackage{geometry}
\usepackage{amsthm}

\usepackage[colorlinks]{hyperref}
\usepackage{cleveref}

\hypersetup{allcolors=[RGB]{220 20 60}}

\newtheorem{theorem}{Theorem}

\newtheorem{proposition}{Proposition}
\newtheorem{assumption}{Assumption}
\newtheorem{definition}{Definition}
\newtheorem{remark}{Remark}

\crefformat{equation}{(#2#1#3)}
\crefname{assumption}{assumption}{assumptions}

\usepackage{eucal}
\usepackage{dsfont}
\usepackage{xspace}
\usepackage{xifthen}

\newcommand{\set}[1]{\mathbb{#1}}
\newcommand{\func}[1]{\mathds{#1}}
\newcommand{\oper}[1]{\mathcal{#1}}
\newcommand{\dist}[1]{\mathcal{#1}}

\DeclareMathOperator*{\minimize}{minimize}

\newcommand{\KL}{\func{D}}

\newcommand{\Cov}{\operatorname{Cov}}

\newcommand{\transp}[1]{{#1}^{\mathsf{T}}}

\newcommand{\R}{\set{R}}

\newcommand{\E}{\func{E}}
\newcommand{\F}{\func{F}}
\newcommand{\Z}{\func{Z}}

\newcommand{\scd}[1][]{\ifthenelse{\isempty{#1}}%
        {\textsc{Scd}}
        {\textsc{Scd{#1}}}
        }

\newcommand{\uline}[1]{\underline{#1}}

\newcommand{\opt}[1]{{#1}^{\star}}
\newcommand{\inv}[1]{{#1}^{-1}}

\newcommand{\define}{\coloneqq}

\newcommand{\MI}{\func{I}}

\newcommand{\Fbeta}{F_\beta}

\newcommand{\sspace}{\set{X}}
\newcommand{\cspace}{\set{U}}
\newcommand{\ospace}{\set{Y}}

\newcommand{\Jtrj}{J_{\mathrm{trj}}}
\newcommand{\JJtrj}{\oper{J}_{\mathrm{trj}}}

\newcommand{\seq}{\boldsymbol}
\newcommand{\causals}{\Pi}

\newcommand{\Jirr}{J_{\mathrm{irr}}}

\newcommand{\Jol}{J_{\mathrm{ol}}}
\newcommand{\Jsc}{J_{\mathrm{sc}}}

\makeatletter
\def\ps@IEEEtitlepagestyle{%
  \def\@oddfoot{\mycopyrightnotice}%
  \def\@evenfoot{}%
}
\def\mycopyrightnotice{%
  {\hfill \footnotesize 978-1-4673-9563-2/15/\$31.00 \copyright 2015 IEEE\hfill}
}
\makeatother

\title{\LARGE \bf
Fundamental Limits for Sensor-Based Control via the \\ Gibbs Variational Principle
}

\author{Vincent Pacelli and Evangelos A. Theodorou
\thanks{This article has been accepted for publication in IEEE Control Systems Letters. \textcopyright2026 IEEE. DOI \texttt{10.1109/LCSYS.2026.3710535}}
\thanks{The authors are associated with the Daniel Guggenheim School of Aerospace Engineering, Georgia Institute of
Technology, Atlanta, GA, USA, 30308.
        \mbox{\tt\small \{vpacelli, evangelos.theodorou\}@gatech.edu}}%
}

\usepackage{tikz}

\newcommand\copyrighttext{%
  \footnotesize \textcopyright \the\year{} IEEE. This article has been accepted for publication in IEEE Control Systems Letters. This is the author's version and has not been fully edited. Content may change prior to publication. DOI \texttt{10.1109/LCSYS.2026.3710535.}}

\begin{document}

\maketitle

\thispagestyle{empty}
\pagestyle{empty}

\begin{abstract}
Fundamental limits on the performance of feedback controllers are essential for benchmarking algorithms, guiding sensor selection, and certifying task feasibility---yet few general-purpose tools exist for computing them. Existing information-theoretic approaches overestimate the information a sensor must provide by evaluating it against the uncontrolled system, producing bounds that degrade precisely when feedback is most valuable. We derive a lower bound on the minimum expected cost of any causal feedback controller under partial observations by applying the Gibbs variational principle to the joint path measure over states and observations. The bound applies to nonlinear, nonholonomic, and hybrid dynamics with unbounded costs and admits a self-consistent refinement: any good controller concentrates the state, which limits the information the sensor can extract, which tightens the bound. The resulting fixed-point equation has a unique solution computable by bisection, and we provide conditions under which the free energy minimization is provably convex, yielding a certifiably correct numerical bound. On a scalar LQG problem the self-consistent bound captures over 80\% of the known optimal cost at moderate sensor noise, and on a nonlinear Dubins car tracking problem it remains informative across all noise levels where a bound using the uncontrolled state distribution is vacuous.
\end{abstract}

\section{Introduction}
\label{sec:intro}
Feedback control systems must operate with imperfect information---even sensors of the highest quality produce noisy and incomplete observations. Measurement uncertainty places a \emph{fundamental limit} on the performance any feedback controller can achieve, regardless of its computational sophistication. Characterizing these limits allows engineers to benchmark controllers against what is physically achievable, guide sensor selection and placement, and certify whether a task is feasible with a given sensing modality.

Few general-purpose tools exist for computing such bounds. Classical
results---the Nyquist--Shannon sampling theorem~\cite{Shannon49},
Bode's sensitivity integral~\cite{Bode45}, data-rate theorems for
networked control~\cite{Nair07}, and perception-latency limits for agile flight~\cite{Falanga19}---each address a particular problem class. The task-relevant information potential (TRIP) framework~\cite{Majumdar22, Majumdar23} provides the first general-purpose approach, lower-bounding the optimal cost of any output-feedback controller via a generalization of Fano's inequality~\cite{Gerchinovitz20}. While broadly applicable, the TRIP bound requires bounded costs and a finite control space. More fundamentally, it computes mutual information (MI) using the
\emph{open-loop} state distribution---which arises in the absence of feedback. In most instances, good feedback controllers concentrate the state in low-cost regions, meaning the open-loop distribution vastly overestimates the required sensor information. This structural issue is worst when feedback control is the most effective.

The connection between stochastic optimal control and statistical mechanics has been exploited to \emph{compute} optimal controllers under structural assumptions. In particular, free energy duality and the Kullback-Leibler (KL) divergence play a central role in linearly-solvable optimal control~\cite{Todorov06}, path integral control~\cite{Kappen05, Theodorou12}, risk-sensitive formulations~\cite{DaiPra96}, fully probabilistic control design~\cite{Karny96, Karny05}, and thermodynamic analyses of stochastic control~\cite{Chen20}. We observe that the same structure---specifically, the Gibbs
variational principle (GVP)---can be used to \emph{bound} the optimal
cost. The TRIP bound is a special case of the GVP applied to the
joint path measure over states and observations. This connection has three consequences: (i) the bound extends immediately to unbounded costs and continuous control spaces, (ii) the free energy retains the full cost distribution through the partition function rather than collapsing it to a binary threshold, and is therefore never looser, and (iii) most importantly, it reveals a self-consistency structure invisible in the Fano-based derivation. Any controller that achieves low cost concentrates the state---limiting the value of sensory information and tightening the bound. Specifically, this paper contributes:
\begin{enumerate}
    \item A \textbf{task-relevant free energy (TRFE) lower bound} (\Cref{thm:trfe}) on the
    minimum expected cost of any causal feedback controller operating
    under partial observations. The bound applies to nonlinear,
    nonholonomic, and hybrid dynamics with unbounded costs and
    continuous control spaces.

    \item A \textbf{self-consistent refinement} (\Cref{thm:sc}) that
    exploits the coupling between achievable cost and sensor
    information. The resulting fixed-point equation has a unique
    solution that improves upon the open-loop estimate.

    \item \textbf{Conditions for certifiable computation}
    (\Cref{thm:convexity}) under which the free energy is strictly
    convex, guaranteeing that the numerical bound is a certified lower
    bound on the optimal cost. 
\end{enumerate}
We validate the bound on a scalar LQG problem where the optimal cost is known in closed form and on a nonlinear Dubins car tracking problem (\Cref{sec:example}), where the self-consistent bound remains informative across all noise levels while a bound using the uncontrolled state distribution is vacuous.

\section{Problem Formulation}
\label{sec:problem}
Denote by $\sspace = \set{R}^n$ and $\cspace \subseteq \set{R}^m$ the state and control spaces, respectively. The system evolves according to discrete-time nonlinear dynamics:
\begin{align}
    x_{t + 1} = f(x_t, u_t, w_t), && (x_0, w_{0:T}) \sim P_w.\label{eq:dyn}
\end{align}
Here, $P_w$ is the joint distribution of the initial condition and process noise. Let $\ospace = \set{R}^p$ be the output space of the sensor,
\begin{align}
    y_t = h(x_t) + v_t, && v_t \sim \dist{N}(0, \Sigma_v),\label{eq:sense}
\end{align}
where measurement noise $v_t$ is sampled i.i.d.\ and independent of the process noise $w_t$.\footnote{Throughout,
assume distributions admit densities with respect to a common measure, functions are measurable, and all expectations are finite.}

A non-negative cost function \mbox{$c_t: \sspace \times \cspace \to \R_+$} defines the \emph{control task}. We make two assumptions on its structure:
\begin{assumption}[Separability]
\label{ass:separability}
The cost function decomposes $c_t(x, u)$ into a state cost and a non-negative control cost: i.e., $c_t(x, u) = q_t(x) + r_t(u)$.
\end{assumption}

\begin{assumption}[Coercivity]
\label{ass:coercivity}
The state cost is coercive, namely: \mbox{$q_t(x) \geq a \norm{x}^2 - b$} for constants \mbox{$a > 0$, $b \geq 0$}.
\end{assumption}
Separability is used throughout to ensure the control cost is deterministic given $\seq{u}$, which is essential for the free energy factorization in \Cref{thm:trfe}, and to allow the state cost to define a cost-information certificate (\Cref{sec:self-consistent}). Coercivity ensures that a finite cost budget constrains the second moment of the state distribution---which is essential for establishing a self-consistent bound (\Cref{sec:self-consistent}).

Write $\seq{\xi} = (x_0, w_{0:T-1})$ for a realization of the stochastic primitives. The \emph{trajectory cost} under an open-loop control sequence $\seq{u} \in \cspace^T$ is
\begin{align}
    \Jtrj(\seq{u}; \seq{\xi}) \define\ &\Jtrj^x(\seq{u}; \seq{\xi}) + \Jtrj^u(\seq{u})\\
    \Jtrj^x(\seq{u}; \seq{\xi}) \define \sum_{t = 0}^T q_t(x_t), &\qquad \Jtrj^u(\seq{u}) \define \sum_{t = 0}^T r_t(u_t),\nonumber
\end{align}
where $x_{0:T}$ satisfies \eqref{eq:dyn} given $\seq{u}$ and $\seq{\xi}$. A \emph{causal output-feedback controller} is a sequence $\seq{\pi} = \pi_{0:T}$ where each $\pi_t$ maps the observation history $y_{0:t}$ to a control $u_t \in \cspace$. Denote the set of all such controllers by $\causals$. For $\seq{\pi} \in \causals$, write $\JJtrj[\seq{\pi}]$ for the trajectory cost evaluated using $u_t = \pi_t(y_{0:t})$, with state and control components $\JJtrj^x[\seq{\pi}]$ and $\JJtrj^u[\seq{\pi}]$ defined similarly. Expectations taken using a given $\seq{\pi}$ are written $\E_{\seq{\pi}}[\cdot]$.

The object of study is the optimal control problem:
\begin{align}
    \opt{J} \define \minimize_{\seq{\pi} \in \causals}\quad \E_{\seq{\pi}}\qty\Big[\JJtrj[\seq{\pi}]]\eqqcolon \func{J}[\seq{\pi}].\tag{OPT}\label{eq:cl_opt}
\end{align}
We seek a lower bound on $\opt{J}$ computable from the problem data alone.\footnote{No assumptions beyond density regularity are placed on $P_w$; see \Cref{thm:trfe} for the precise integrability condition.}

\section{Task-Relevant Free Energy Bound}
\label{sec:trfe}
The central idea is to compare the feedback controller's performance
against a reference measure \emph{in which sensor outputs are decoupled from
the state}. The KL divergence between the true and reference path
measures quantifies the information the controller utilizes, and the
Gibbs variational principle (GVP) implies that the cost reduction
relative to the reference cannot exceed this information.\footnote{Unlike standard variational formulations in risk-sensitive control~\cite{Theodorou12, DaiPra96}, where the optimization variable appears on the left of the KL divergence, here the true closed-loop measure $P$ appears on the left and the reference $Q$---which contains the free optimization variable $\nu_t$---on the right. This reversal arises because we bound the cost from below rather than solve for the optimal controller.} For any measurable $\ell : \set{X} \to \set{R}$ and $\beta > 0$,
define the \emph{partition function}
$\Z_\beta[\ell] \define \E_Q[\exp(-\beta \ell)]$ and the \emph{free
energy} $\F_\beta[\ell] \define -\frac{1}{\beta}\log \Z_\beta[\ell]$.
Then, for $P \ll Q$:
\begin{align}
     \F_\beta[\ell] - \frac{1}{\beta} \KL[P \| Q] &\leq \E_P[\ell],
     \tag{GVP}\label{eq:gvp}
\end{align}
with equality if and only if $P(x) / Q(x) \propto \exp(-\beta \ell)$. The
parameter $\beta > 0$ acts as an inverse temperature, interpolating
between an average-case bound ($\beta \to 0^+$) and a worst-case
bound ($\beta \to \infty$); optimizing over $\beta$ yields the
tightest result.

The following bound requires a \emph{uniform mutual information (MI) budget}
$\bar{I} \geq 0$ satisfying
$\sum_{t=0}^{T} \MI_{\seq{\pi}}[x_t; y_t] \leq \bar{I}$ for all
$\seq{\pi} \in \causals$,
where $\MI_{\seq{\pi}}[x_t; y_t] \define \E_{\seq{\pi}}\!\big[\log \frac{p(x_t, y_t)}{p(x_t)\,p(y_t)}\big]$ is the mutual information between state and observation at time $t$;
\Cref{sec:self-consistent} constructs such a budget from the cost structure.

\begin{theorem}[TRFE Bound]
    \label{thm:trfe}
    Under \Cref{ass:separability}, let $\bar{I} \geq 0$ be a uniform MI budget. Then,
    \begin{align}
        \Jirr(\bar{I}) \define \sup_{\beta > 0}\qty{\inf_{\seq{u} \in \cspace^T} F^x_\beta(\seq{u}) + \Jtrj^u(\seq{u}) - \frac{\bar{I}}{\beta}} \leq \opt{J},\label{eq:trfe}
    \end{align}
    where $F^x_\beta(\seq{u}) \define -\frac{1}{\beta}\log \E_{\seq{\xi}}\qty[\exp\!\big({-\beta\, \Jtrj^x(\seq{u};\, \seq{\xi})}\big)]$ is the free energy of the state cost.
\end{theorem}

\begin{proof}
    Fix $\seq{\pi} \in \causals$ and $\beta > 0$. Write $\tau = (x_{0:T},\, y_{0:T})$ for a trajectory. Define two distributions over $\tau$:
     \begin{align}
        P^{\seq{\pi}}(\tau) &= p_0(x_0) \prod_{t=0}^{T-1} p(x_{t+1} \mid x_t, u_t) \prod_{t=0}^{T} p(y_t \mid x_t),\label{eq:cl_dist}\\
        Q^{\seq{\pi}}(\tau) &= p_0(x_0) \prod_{t=0}^{T-1} p(x_{t+1} \mid x_t, u_t) \prod_{t=0}^{T} \nu_t(y_t),\nonumber
    \end{align}
    where $u_t = \pi_t(y_{0:t})$ in both cases and each $\nu_t(y_t)$ is an arbitrary density over $\ospace$. Applying \eqref{eq:gvp} with $\ell = \JJtrj[\seq{\pi}]$:
    \begin{align}
    \func{J}[\seq{\pi}] \geq -\tfrac{1}{\beta}\log \E_{Q^{\seq{\pi}}}[\exp\!\big(\!{-\beta\, \JJtrj[\seq{\pi}]}\big)] - \tfrac{1}{\beta}\KL[P^{\seq{\pi}} \| Q^{\seq{\pi}}].\label{eq:gvp_applied}
    \end{align}
    Since $P^{\seq{\pi}}$ and $Q^{\seq{\pi}}$ share the same dynamics, their KL divergence reduces to $\sum_{t=0}^{T} \E_{\seq{\pi}}\big[\KL[p(y_t \mid x_t) \| \nu_t(y_t)]\big]$, which is minimized at $\opt{\nu}_t = p^{\seq{\pi}}(y_t)$, yielding:
    \begin{align}
        \inf_{\nu_{0:T}} \KL\qty[P^{\seq{\pi}} \| Q^{\seq{\pi}}] = \textstyle\sum_{t=0}^{T} \MI_{\seq{\pi}}[x_t; y_t] \leq \bar{I}.\label{eq:kl_mi}
    \end{align}
    Under $Q^{\seq{\pi}}$, the observations are independent of the system state. For each fixed realization of $y_{0:T}$ the control sequence \mbox{$\seq{u}(y_{0:T}) = \pi(y_{0:T})$} is deterministic. The only remaining randomness in the state cost is due to the process noise $\seq{\xi}$. The expectation in~\eqref{eq:gvp_applied} factors as
    \begin{align}
        &\E_{Q^{\seq{\pi}}}\!\big[\exp(-\beta\, \JJtrj[\seq{\pi}])\big]\label{eq:factor}\\
        &\quad= \E_{y_{0:T}}\!\Big[\exp(-\beta\, \Jtrj^u(\seq{u}(y_{0:T})))\,
        \Z_\beta\!\big[\Jtrj^x(\seq{u}(y_{0:T});\,\cdot\,)\big]\Big].
        \nonumber
    \end{align}
    Since $F_\beta^x(\seq{u}(y_{0:T})) + \Jtrj^u(\seq{u}(y_{0:T})) \geq \inf_{\seq{u}}\{F_\beta^x(\seq{u}) + \Jtrj^u(\seq{u})\}$ for any process noise realization, the logarithm in~\eqref{eq:factor} is bounded above by $\exp\!\big({-\beta \inf_{\seq{u}}\{F_\beta^x(\seq{u}) + \Jtrj^u(\seq{u})\}}\big)$. Apply the decreasing function $-\inv{\beta}\log(\cdot)$ to reverse the inequality:
    \begin{align}
        -\tfrac{1}{\beta}\log \E_{Q^{\seq{\pi}}}\![\exp\!\big(\!{-\beta\, \JJtrj[\seq{\pi}]}\big)] &\geq \underset{\seq{u}}{\inf} F_\beta^x(\seq{u}) + \Jtrj^u(\seq{u}).\label{eq:fe_bound}
    \end{align}
    
    Substituting~\eqref{eq:kl_mi} and~\eqref{eq:fe_bound} into~\eqref{eq:gvp_applied} gives:
    \begin{align}
        \func{J}[\seq{\pi}] \geq \inf_{\seq{u}}\{F_\beta^x(\seq{u}) + \Jtrj^u(\seq{u})\} - \bar{I}/\beta.
    \end{align} Since this holds for all $\seq{\pi} \in \causals$ and $\beta > 0$, optimizing over these free variables yields~\eqref{eq:trfe}.
\end{proof}

\begin{remark}
The TRIP bound~\cite{Majumdar22,Majumdar23} is a special case: replacing $\Jtrj^x$ with $\mathbf{1}[\Jtrj^x > \theta]$ and taking $\beta \to \infty$ recovers the generalized Fano inequality.
\end{remark}

The function $\Jirr(\bar{I})$ in \eqref{eq:trfe} is the \emph{irreducible task cost}: the minimum cost remaining after accounting for the sensor's information budget.\footnote{More precisely, $\Jirr(\bar{I})$ is a lower bound on the irreducible cost, since tighter bounds on $\opt{J}$ may exist.} \Cref{sec:convexity} gives conditions under which the inner minimization over $\seq{u}$ is tractable. The following properties of $\Jirr(\bar{I})$ are used in \Cref{sec:self-consistent} to establish the self-consistent fixed-point bound.

\begin{proposition}
\label{prop:irr}
Define the \emph{optimal open-loop task cost},
\begin{align}
    \Jol \define \inf_{\seq{u} \in \cspace^T}\; \E_{\seq{\xi}}\!\big[\Jtrj(\seq{u};\, \seq{\xi})\big].\label{eq:j_ol}
\end{align}
When $\Jol > 0$, $\Jirr(\bar{I})$ has the following properties for \mbox{$\bar{I} \geq 0$}: (i)~$\Jirr(\bar{I})$ is strictly decreasing, (ii)~\mbox{$\Jirr(0) = \Jol$}, and (iii)~\mbox{$\Jirr(\bar{I}) \leq \Jol$}.
\end{proposition}
\begin{proof}
For fixed $\beta > 0$ and $\seq{u}$, the map $\bar{I} \mapsto F_\beta^x(\seq{u}) + \Jtrj^u(\seq{u}) - \bar{I}/\beta$ is strictly decreasing in $\bar{I}$---a property preserved by the infimum over $\seq{u}$ (of which $\bar{I}$ is independent) and the supremum over $\beta$. The identity $\Jirr(0) = \Jol$ holds as a consequence of $F_\beta^x(\seq{u}) \to \E_{\seq{\xi}}[\Jtrj^x(\seq{u}; \seq{\xi})]$ as $\beta \to 0^+$ (L'H\^{o}pital's rule).
\end{proof}

It remains to identify a valid MI budget $\bar{I}$. The next section shows that the cost structure itself provides one: any controller that achieves low cost must keep the state concentrated, which limits the MI the sensor can carry.

\section{Tightening by Coupling Task Cost and Sensor Information}
\label{sec:self-consistent}
The TRFE bound (\Cref{thm:trfe}) requires a uniform MI budget
$\bar{I} \geq 0$. Identifying one is nontrivial: different
controllers induce different trajectory distributions and therefore
different MI values for the same sensor. However, the MI a sensor
can extract is constrained by the cost the controller achieves---a
performant controller concentrates the state distribution which bounds the MI. This connection between the MI and task cost is formalized in this section.

\subsection{Cost-Information Certificates}
\label{sec:sc_info}

Under \Cref{ass:separability}, $r_t \geq 0$ implies that any policy with $\func{J}[\seq{\pi}] \leq J$ satisfies $\sum_{t=0}^{T} \E_{\seq{\pi}}[q_t(x_t)] \leq J$. Combined with \Cref{ass:coercivity}, this constrains the uncentered second moments $M_t \define \E_{\seq{\pi}}[x_t \transp{x_t}]$:
\begin{align}
    a \textstyle\sum_{t=0}^{T} \tr(M_t) \leq J + b\,(T + 1). \label{eq:moment_bound}
\end{align}
When $q_t(x) = \transp{x}Qx$ with $Q \succ 0$, this sharpens to $\sum_{t=0}^{T} \tr(QM_t) \leq J$. Since the Gaussian maximizes differential entropy for a fixed covariance~\cite[Theorem~8.6.5]{Cover05}, the MI between state and observation is bounded by
\begin{align}
    \MI_{\seq{\pi}}[x_t;\, y_t]
    \leq \frac{1}{2}\log\det\!\big(I_p + \inv{\Sigma}_v\,
    \Sigma_{h_t}\big), \label{eq:mi_bound}
\end{align}
where $\Sigma_{h_t} = \Cov(h(x_t))$. These two facts motivate an abstract definition of an information budget compatible with a given cost.

\begin{definition}[Cost-Information Certificate]
    \label{def:ci_cert}
    A function $\bar{I} : [0,\infty) \to [0,\infty)$ is a
    \emph{cost-information (CI) certificate} for
    system~\cref{eq:dyn} and \cref{eq:sense} if:
    \begin{enumerate}
        \item[(a)] $\bar{I}$ is continuous and strictly increasing; and
        \item[(b)] for $\seq{\pi} \in \causals$ with
            $\func{J}[\seq{\pi}] \leq J$,\;
            $\sum_{t=0}^{T} \MI_{\seq{\pi}}(x_t;\, y_t)
            \leq \bar{I}(J)$.
    \end{enumerate}
\end{definition}

The following propositions construct certificates under two
sensor models.

\begin{proposition}[Nonlinear Sensors]
    \label{prop:nonlinear_mi}
    Let $h(x)$ be $L$-Lipschitz with $L > 0$, and let $\lambda_{\min}$
    denote the smallest eigenvalue of $\Sigma_v$. Under
    \Cref{ass:separability,ass:coercivity}, the function
    \begin{align}
        \bar{I}(J) \define \frac{(T + 1)p}{2}\log\!\left(
        1 + \frac{L^2(J + bT)}{ap\,(T + 1)\,\lambda_{\min}}\right)
        \label{eq:ibar_nonlinear}
    \end{align}
    is a cost-information certificate for
    system~\eqref{eq:dyn}--\eqref{eq:sense}.
\end{proposition}
\begin{proof}
    Applying the arithmetic mean--geometric mean (AM-GM) inequality to~\cref{eq:mi_bound},\footnote{Specifically,
    for $A \in \set{S}^n_+$, $\det(I_p + A) \leq (1 + \tr(A)/p)^p$.}
    and using $\tr(\inv{\Sigma}_v \Cov(h(x_t))) \leq L^2\tr(M_t)/\lambda_{\min}$:
    \begin{align}
        \MI_{\seq{\pi}}(x_t;\, y_t)
        \leq \frac{p}{2}\log\!\left(1 + \frac{L^2\,\tr(M_t)}{p\,\lambda_{\min}}\right).
        \label{eq:mi_nonlinear_step}
    \end{align}
    Summing over $t = 0, \dots, T$, applying Jensen's inequality to $\log(\cdot)$, and substituting~\cref{eq:moment_bound}
    yields property~(b). Property~(a) follows from the strict
    monotonicity of $\log(\cdot)$ and $L > 0$.
\end{proof}

When the sensor is linear and the state cost quadratic, both
relaxations tighten:

\begin{proposition}[Linear Sensors]
    \label{prop:linear_mi}
    Let $h(x) = Hx$ with $\inv{\Sigma}_v H \neq 0$, and let
    $q_t(x) = \transp{x}Qx$ with $Q \in \set{S}^n_{+}$, $Q \neq 0$. Under
    \Cref{ass:separability}, the function
    \begin{maxi}|s|{M_t \succeq 0}
        {\sum_{t=0}^{T} \frac{1}{2}\ln\det\!\big(I_p +
        \Sigma_v^{-1} H M_t \transp{H}\big)}
        {\label{eq:mi_opt}}{\bar{I}(J) \define}
        \addConstraint{\tr(QM_t)}{\leq J}
    \end{maxi}
    is a cost-information certificate for
    system~\eqref{eq:dyn}--\eqref{eq:sense}.
\end{proposition}
\noindent Program~\eqref{eq:mi_opt} is convex and admits a closed-form
solution via water-filling across the modes of
$Q^{-1/2}\transp{H}\inv{\Sigma}_v H Q^{-1/2}$, with the optimal
second-moment allocation uniform over time~\cite{Cover05}. Property~(b) of \Cref{def:ci_cert} follows from
the feasibility of $M_t^{\seq{\pi}}$ in~\eqref{eq:mi_opt} for any
$\seq{\pi}$ with $\func{J}[\seq{\pi}] \leq J$ and
\Cref{eq:mi_bound}; property~(a) from the strict monotonicity of
$\ln\det(\cdot)$ and the assumption $\inv{\Sigma}_v H \neq 0$.

\subsection{Tightening via a Fixed-Point Constraint}
\label{sec:sc_bound}

Define $\Phi(J) \define (\Jirr \circ \bar{I})(J)$, where $\bar{I}(J)$
is a certificate. Since the optimal policy achieves cost $\opt{J}$,
applying \Cref{thm:trfe} with budget $\bar{I}(\opt{J})$ gives
$\opt{J} \geq \Phi(\opt{J})$. This suggests searching for the largest
$J$ consistent with this constraint, i.e., the fixed point
$\Jsc = \Phi(\Jsc)$.

\begin{theorem}[Self-Consistent TRFE Bound]
    \label{thm:sc}
    Let $\bar{I}$ be a certificate for
    system~\eqref{eq:dyn}--\eqref{eq:sense} with $\bar{I}(\Jol) > 0$,
    and suppose $\Jol > 0$. Then:
    \begin{enumerate}
        \item[(i)] The function $\Phi(J)$ is continuous and strictly
        decreasing on $[0, \Jol]$.
        
        \item[(ii)] There is a unique solution
        $\Jsc \in (0,\, \Jol)$ to the fixed-point equation:
        \mbox{$J = \Phi(J)$}.
        
        \item[(iii)] This solution lower-bounds the optimal cost:
        \mbox{$\Jsc \leq \opt{J}$}.
    \end{enumerate}
\end{theorem}
\begin{proof}
We prove each property using $g(J) \define \Phi(J) - J$.

\begin{enumerate}
\setlength{\itemsep}{3pt}
    \item[(i)] $\bar{I}$ is continuous and strictly increasing by
    \Cref{def:ci_cert}(a), and $\Jirr(\bar{I})$ is continuous and
    strictly decreasing (\Cref{prop:irr}), so their composition
    $\Phi(J)$ is as well. Thus $g$ is continuous with
    $g'(J) < -1$ where the derivative exists.

    \item[(ii)] By \Cref{prop:irr}(ii),
    $\Jirr(0) = \Jol$, so $g(0) = \Jol > 0$. Since
    $\bar{I}(\Jol) > 0$, we have
    $\Phi(\Jol) < \Jirr(0) = \Jol$ and thus $g(\Jol) < 0$.
    The intermediate value theorem guarantees a root
    $\Jsc \in (0, \Jol)$, unique by~(i).

    \item[(iii)] The optimal policy achieves cost $\opt{J}$, so
    \Cref{def:ci_cert}(b) gives $\bar{I}(\opt{J})$ as a valid MI
    budget for \Cref{thm:trfe}: $\opt{J} \geq \Phi(\opt{J})$, i.e.,
    $g(\opt{J}) \leq 0.$ Since $g$ is strictly decreasing and
    $g(\Jsc) = 0$, it follows that $\opt{J} \geq \Jsc$. \qedhere
\end{enumerate}
\end{proof}

\Cref{thm:sc} holds for any CI certificate satisfying \Cref{def:ci_cert}, including the nonlinear bound~\eqref{eq:ibar_nonlinear}. The linear certificate~\eqref{eq:mi_opt} yields a tighter result but is not required for convergence. The self-consistent bound is strictly tighter than substituting the
open-loop MI budget directly into \Cref{thm:trfe}: since
$\Jsc < \Jol$ and $\bar{I}$ is strictly increasing,
$\bar{I}(\Jsc) < \bar{I}(\Jol)$, and therefore
$\Jsc = \Phi(\Jsc) > \Phi(\Jol)$.

\section{Convexity of the Free Energy}
\label{sec:convexity}
Computing \cref{eq:trfe} requires solving a \emph{global} minimization over
open-loop control sequences. This section gives conditions under which
$\Fbeta(\seq{u}) \define F_\beta^x(\seq{u}) + \Jtrj^u(\seq{u})$ is strictly convex in $\seq{u}$, so that any
critical point is a global minimum and the computed bound is
certifiably valid.\footnote{If these conditions are not met,
\Cref{thm:trfe,thm:sc} remain valid. Convexity is needed only to
certify the numerical computation of
$\inf_{\seq{u}} \Fbeta(\seq{u})$.}

\begin{assumption}[Matched Actuator Noise and Control Cost]
\label{ass:matched}
The process noise enters through the control input channel:
\begin{align}
    x_{t+1} = f(x_t,\, u_t + w_t),\ 
    w_t \sim \dist{N}(0, \Sigma_w),\ \Sigma_w \succ 0, \label{eq:matched}
\end{align}
with $f$ twice differentiable, so that the state trajectory depends
on $\seq{u}$ and $\seq{w}$ only through the realized input
$\seq{\zeta} \define \seq{u} + \seq{w}$. The control cost is
quadratic: $r(u) \define \tfrac{1}{2}\transp{u}\, R\, u$, $R \succ 0$.
\end{assumption}

Under \Cref{ass:matched}, the trajectory state cost depends on
$\seq{u}$ and $\seq{w}$ only via the realized input,
\mbox{$S(\seq{\zeta}) \define \sum_{t=0}^{T} q(x_t(\seq{\zeta}))$},
where $x_t(\seq{\zeta})$ is the state at time $t$ under input
sequence $\seq{\zeta}$ from a fixed initial condition $x_0$.

\begin{assumption}[Semiconvexity]
\label{ass:semiconvexity}
The trajectory cost $S$ is $\alpha$-\emph{semiconvex}
\cite{Cannarsa04} in $\seq{\zeta}$: there exists a value
$\alpha \geq 0$ such that
$\partial_{\seq{\zeta}}^2 S(\seq{\zeta}) \succeq -\alpha\, I_{mT}$
for all $\seq{\zeta} \in \R^{mT}$.
\end{assumption}

\begin{theorem}[Convexity of the Free Energy]
\label{thm:convexity}
Under \Cref{ass:separability,ass:matched,ass:semiconvexity},
when $\alpha < \uline{R} \define \lambda_{\min}(R)$,
$\Fbeta(\seq{u})$ is strictly convex for
$\beta \in (0,\, \bar{\beta})$, where
\begin{align}
    \bar{\beta} \define \frac{\uline{R} - \alpha}
    {\uline{R}\,\alpha\,\bar{\sigma}^2}, \qquad
    \bar{\sigma}^2 \define \lambda_{\max}(\Sigma_w).
    \label{eq:betamax}
\end{align}
When $\alpha = 0$ (i.e., $S$ is convex in $\seq{\zeta}$),
$\bar{\beta} = +\infty$.
\end{theorem}

\begin{proof}
Let $\Sigma_W \define I_T \otimes \Sigma_w$ and
$P_{\seq{u}} \define \dist{N}(\seq{u}, \Sigma_W)$ be the
distribution of $\seq{\zeta}$ under input $\seq{u}$.
The likelihood ratio
$P_{\seq{u}}/P_0 = \exp(\transp{\seq{u}}\inv{\Sigma}_W \seq{\zeta}
- \tfrac{1}{2}\transp{\seq{u}} \inv{\Sigma}_W \seq{u})$
yields the decomposition
\begin{align}
    F_\beta(\seq{u}) &= \underbrace{-\tfrac{1}{\beta}\log
    \func{E}_{P_0}[\exp(-\beta S(\seq{\zeta})
    + \transp{\seq{u}}\inv{\Sigma}_W
    \seq{\zeta})]}_{-h(\seq{u})}\\
    &\qquad+ \underbrace{\tfrac{1}{2} \transp{\seq{u}}
    \qty(\tfrac{1}{\beta}\inv{\Sigma}_W + R)
    \seq{u}}_{g(\seq{u})}.\nonumber
\end{align}
The quadratic $g$ is strictly convex, and $h$ is convex by
H{\"o}lder's inequality, so $F_\beta = g - h$ is a difference of
convex functions. Convexity of $F_\beta$ follows from
$\partial_{\seq{u}}^2 g \succ \partial_{\seq{u}}^2 h$.

Differentiating $h$ twice, its Hessian is
\begin{align}
    \partial_{\seq{u}}^2 h(\seq{u}) = \tfrac{1}{\beta}
    \inv{\Sigma}_W \mathrm{Cov}_{\mu}[\seq{\zeta}; \seq{u}]
    \inv{\Sigma}_W,
\end{align}
where $\mathrm{Cov}_{\mu}[\seq{\zeta}; \seq{u}]$ is the covariance
of $\seq{\zeta}$ under the tilted measure
$\mu(\seq{\zeta}; \seq{u}) \propto \exp(-\Psi(\seq{\zeta};
\seq{u}))$ with
\begin{align}
    \Psi(\seq{\zeta}; \seq{u}) \define \beta S(\seq{\zeta})
    + \tfrac{1}{2} \transp{\seq{\zeta}}\inv{\Sigma}_W \seq{\zeta}
    - \transp{\seq{u}}\inv{\Sigma}_W \seq{\zeta}.
\end{align}
Under \Cref{ass:semiconvexity},
$\partial_{\seq{\zeta}}^2 \Psi = \beta
\partial_{\seq{\zeta}}^2 S + \inv{\Sigma}_W \succeq
\inv{\Sigma}_W - \alpha \beta I_{mT}$,
which is positive definite when $\alpha \beta \bar{\sigma}^2 < 1$.
The Brascamp--Lieb inequality~\cite[Thm.~4.1]{Brascamp76}\footnote{If
$p(z) \propto \exp(-\ell(z))$ with
$\partial_z^2 \ell \succeq K \succ 0$, then
$\Cov_p(z) \preceq \inv{K}$.} then gives
$\mathrm{Cov}_{\mu}[\seq{\zeta}; \seq{u}] \preceq
\inv{(\inv{\Sigma}_W - \alpha \beta I_{mT})}$.
Substituting into $\partial_{\seq{u}}^2 F_\beta = \partial_{\seq{u}}^2 g - \partial_{\seq{u}}^2 h$
and using the identity $I - \inv{X}Y = \inv{X}(X - Y)$ with
commuting factors:
\begin{align}
    \partial_{\seq{u}}^2 F_\beta(\seq{u})
    &\succeq R - \alpha\inv{\Sigma}_W
    \inv{(\inv{\Sigma}_W - \alpha\beta I_{mT})}\nonumber\\
    &\succeq \qty(\uline{R} - \frac{\alpha}
    {1 - \alpha\beta\bar{\sigma}^2}) I_{mT} \succ 0
    \iff \beta < \bar{\beta}.\nonumber
\end{align}
The last step uses the identity,
\begin{align}
    \lambda_{\max}\!\big(\inv{\Sigma}_W
\inv{(\inv{\Sigma}_W - \beta\alpha I)}\big)
= (1-\beta\alpha\bar{\sigma}^2)^{-1},
\end{align} as commuting symmetric
matrices share an eigenbasis.
\end{proof}

Restricting the supremum in~\cref{eq:trfe} to
$(0, \bar{\beta})$ preserves validity—it can only weaken the
bound—and loses no tightness when $\opt{\beta} \leq \bar{\beta}$.

\begin{remark}[Scope of \Cref{ass:matched,ass:semiconvexity}]
\label{rem:scope}
\Cref{ass:matched} covers control-affine systems
$x_{t+1} = f_0(x_t) + B(x_t)(u_t + w_t)$ with actuator noise,
including the Dubins car in \Cref{sec:dubins_car}. It excludes
unmatched disturbances that enter independently of the control
input. \Cref{ass:semiconvexity} holds with $\alpha = 0$ for
linear dynamics with convex costs. For nonlinear dynamics with
convex state cost, $\alpha$ depends on the dynamics curvature
$\partial^2 f$ and is small near a tracking reference; it is
estimable from open-loop rollouts (see \Cref{sec:dubins_car}).

\end{remark}

\section{Numerical Example}
\label{sec:example}
\begin{figure*}[t]
    \centering
    \begin{minipage}[t]{0.48\textwidth}
        \centering
        \includegraphics[width=\textwidth]{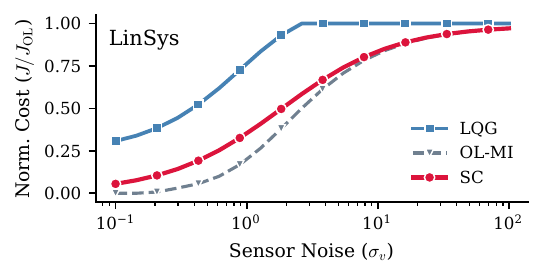}
    \end{minipage}\hfill
    \begin{minipage}[t]{0.48\textwidth}
        \centering
        \includegraphics[width=\textwidth]{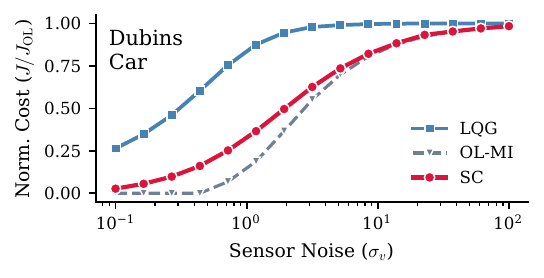}
    \end{minipage}
    \vspace{-1em}
    \caption{Normalized cost vs.\ sensor noise. \textbf{Left:} Scalar
    LQG (optimal cost $\opt{J}$ from Riccati equation).
    \textbf{Right:} Dubins car tracking a figure-eight reference
    (heading unobserved). All costs reported as per-step rates $J/T$
    normalized by $\Jol/T$. \textbf{LQG.} Optimal (left) or
    EKF+LQR (right; upper bound on $\opt{J}$). \textbf{SC.}
    Self-consistent TRFE bound (\Cref{thm:sc}, lower bound on
    $\opt{J}$). \textbf{OL-MI.} TRFE bound with open-loop MI budget.
    The SC bound captures a substantial fraction of $\opt{J}$ at
    moderate-to-high noise in both settings; OL-MI is substantially
    looser at all noise levels.}
    \label{fig:bounds}
    \vspace{-1.75em}
\end{figure*}

Computing the bound requires estimating the free energy
$\Fbeta(\seq{u})$ over a grid of inverse temperatures and then
solving the self-consistent fixed point. The partition function in
$\Fbeta$ is estimated via $N$ Monte Carlo rollouts of the open-loop
dynamics, drawn once and reused across all $\beta$ values and
bisection iterates. Under \Cref{ass:matched}, the minimization over
$\seq{u}$ reduces to an importance-weighted fixed-point
iteration---a deterministic analogue of
MPPI~\cite{Williams18}---that converges to the global minimum when
$\beta < \bar{\beta}$ (\Cref{thm:convexity}). 

\Cref{alg:trfe}
summarizes the procedure. The noise samples $\seq{\xi}^{(1:N)}$ are drawn once and reused across all inverse temperatures and bisection iterates, ensuring that $\Phi(J)$ is a deterministic function of $J$ for stable root-finding. For the open-loop MI bound, $\Phi$ is evaluated at the fixed budget
$\bar{I}_{\mathrm{ol}} = \bar{I}(\Jol)$. For the self-consistent
bound, bisection over $[0, \Jol]$ finds the root of
$J = \Phi(J)$ with $\bar{I}(J)$ given by the
water-filling certificate~\eqref{eq:mi_opt}.

The convexity threshold $\bar{\beta}$ depends on the semiconvexity
constant $\alpha$ (\Cref{ass:semiconvexity}). We estimate $\alpha$ by sampling $N_\alpha$ noise realizations $\seq{\xi}^{(i)}$, computing the Hessian $\nabla^2_{\seq{\zeta}} S$ at each $\seq{\zeta}^{(i)} = \seq{u}_{\mathrm{nom}} + \seq{\xi}^{(i)}$ via automatic differentiation, and setting $\hat{\alpha} = \max(0,\, -\min_i \lambda_{\min}^{(i)})$. Because $\hat{\alpha}$ maximizes over finitely many sampled inputs $\seq{\zeta}^{(i)}$, while the true $\alpha$ is a supremum over $\seq{\zeta} \in \R^{mT}$, the estimate satisfies $\hat{\alpha} \leq \alpha$ and the resulting $\hat{\bar{\beta}} \geq \bar{\beta}$. The bound (\Cref{thm:trfe,thm:sc}) is valid for any $\beta > 0$, but a suboptimal solution to the inner minimization can violate the lower bound guarantee. Convexity certification ensures the computed bound is valid. However, the semiconvexity constant $\alpha$ is estimated from finite samples, and is therefore approximate. A rigorous finite-sample certificate is deferred to future work.

\begin{figure}[t]
\vspace{-0.5em}
\begin{algorithm}[H]
\caption{Estimating the TRFE Bound}
\label{alg:trfe}
\begin{algorithmic}
\REQUIRE Traj. cost $\Jtrj$, CI cert. $\bar{I}$, samples $N$, grid $\{\beta_k\}_{k=1}^K$.
Draw noise $\seq{\xi}^{(1:N)}$.
\FOR{$k = 1, \ldots, K$ {\small(ascending $\beta$ with warm-start)}}
    \STATE $\opt{F}_k \gets \min_{\seq{u}}\{F_{\beta_k}^x(\seq{u}) + \Jtrj^u(\seq{u})\}$.
    \hfill{\small(inner opt.\ in~\eqref{eq:trfe})}
\ENDFOR
\RETURN $\Phi(\bar{I}) \define \max_k\{\opt{F}_k - \bar{I}/\beta_k\}$.\hfill{\small(outer opt.\ in~\eqref{eq:trfe})}
\end{algorithmic}
\end{algorithm}
\vspace{-2.5em}
\end{figure}

\subsection{Scalar Linear-Quadratic-Gaussian}
\label{sec:lqg}

We first validate the bound on a scalar linear--quadratic--Gaussian
(LQG) problem where the optimal cost $\opt{J}$ is known in closed
form. Despite its simplicity, the TRIP bound~\cite{Majumdar22, Majumdar23} \emph{cannot be applied to
this problem}, as it requires both bounded costs and a finite control
space.

The system is,
\begin{align}
    x_{t+1} = 0.95\, x_t + u_t + w_t, && y_t = x_t + v_t,
\end{align}
with $w_t \sim \dist{N}(0, 0.01)$,
$v_t \sim \dist{N}(0, \sigma_v^2)$, $x_0 = 0$, and horizon
$T = 100$. The per-step cost is $c(x,u) = x^2 + u^2$, and the
optimal cost is computed via the discrete-time Riccati equation
(LQG is optimal for this system). For linear-quadratic systems, the free energy is readily computed in closed-form~\cite{whittle1990risk}, so no sampling is necessary.

\Cref{fig:bounds}~(left) shows the normalized cost as a function of
$\sigma_v$. Since $\opt{J}$ is known exactly, the gap between the
LQG and SC curves directly measures the bound's conservatism. The SC
bound captures over half of $\opt{J}$ by $\sigma_v \approx 2$ and
over 80\% by $\sigma_v \approx 10$. As in the Dubins example, the
OL-MI bound is substantially looser at all noise levels.

\subsection{Dubins Car Tracking}
\label{sec:dubins_car}

We validate the bound on a nonlinear, nonholonomic tracking problem
where domain-specific bounds (e.g., Bode integrals) do not apply.

\textbf{System.}
A fixed-velocity Dubins car with state $x = (p_x, p_y, \theta)$
(position and heading) and scalar control $u = \omega$ (turn rate)
tracks a figure-eight reference at constant speed $v = 1$ over
$T = 100$ steps ($\Delta t = 0.1\,\text{s}$). The dynamics are
\begin{align}
    x_{t+1} = x_t + \Delta t \cdot
    \transp{[v\cos\theta_t\ \; v\sin\theta_t\ \; \omega_t + w_t]},
\end{align}
where $w_t \sim \dist{N}(0,\, \sigma_w^2)$, $\sigma_w = 0.1$.
Noise enters only through the steering channel
(\Cref{ass:matched} with $m = 1$).
The sensor observes only position---heading is unobserved:
\mbox{$y_t = Hx_t + v_t$}, \mbox{$H = [I_2 \;\; 0]$},
\mbox{$v_t \sim \dist{N}(0,\, \sigma_v^2 I_2)$}.

\textbf{Task.}
The per-step cost penalizes position tracking error, heading error,
and control effort:
$c(x,u) = \|p - p^{\mathrm{ref}}_t\|^2
+ 0.1\,(1 - \cos(\theta - \theta^{\mathrm{ref}}_t))
+ \tfrac{1}{2}(\omega - \omega^{\mathrm{nom}}_t)^2$,
where $\omega^{\mathrm{nom}}$ is the feedforward steering rate.
The heading cost is smooth and $2\pi$-periodic. For the
water-filling certificate (\Cref{prop:linear_mi}), the quadratic
state cost weight is $Q = \operatorname{diag}(1,\, 1,\, 0)$,
applied to the position components only. Since $Q$ is rank-deficient, the water-filling certificate~\eqref{eq:mi_opt} is computed on the range of $Q$, yielding a valid but potentially conservative MI budget that excludes heading information. The non-quadratic heading
term is retained but does not contribute to the MI certificate. 

\textbf{Bound Computation.}
The free energy is estimated from $N = 5 \times 10^4$ Monte Carlo
rollouts over a grid of $K = 60$ inverse temperatures following
\Cref{alg:trfe}. The estimated semiconvexity constant is
$\hat{\alpha} = 0.44$, yielding $\hat{\bar{\beta}} = 127$. The optimal $\opt{\beta}$ fell below this threshold at all sensor noise levels, providing an approximate finite-sample convexity certificate. All costs are normalized by the open-loop cost $\Jol/T$ so that the range is from $[0, 1]$.\footnote{In the Dubins example, $\Jol$ is approximated by evaluating the sample-average cost at the nominal feedforward input $\seq{u}_{\mathrm{nom}}$. This may overestimate the true infimum, but $\Jol$ enters only as the bisection bracket for \Cref{thm:sc}---the bound values $\{F^{\star}_k\}$ are independently optimized at each $\beta_k$---so the computed $\Jsc$ is unaffected.}

\textbf{Baseline.}
The upper bound on the value of $\opt{J}$ (``LQG'' in \Cref{fig:bounds}~(right)) is the
cost achieved by an extended Kalman filter (EKF) paired with time-varying LQR linearized about the reference, simulated on the nonlinear dynamics. 

\textbf{Results.} \Cref{fig:bounds}~(right) shows the normalized cost as a function of sensor noise $\sigma_v$. The optimal cost $\opt{J}$ lies between the LQG curve (which is an upper bound, since it is one feasible policy) and the SC curve (a lower bound via \Cref{thm:sc}); the gap between them quantifies current uncertainty about $\opt{J}$. At low noise ($\sigma_v = 0.1$), the self-consistent bound accounts for approximately 10\% of the LQG cost while the open-loop MI bound is vacuous (negative), illustrating the importance of the self-consistent mechanism. The SC bound is tighter than OL-MI because it uses an adaptive information budget coupled to the achievable cost, whereas OL-MI uses a fixed budget computed from the uncontrolled state distribution---which vastly overestimates the required sensor information when feedback is effective. As $\sigma_v \to \infty$, all curves converge to the open-loop cost. Observations become uninformative in this limit, so feedback provides no benefit over open-loop control and the bound is tight. The persistent gap between the LQG and SC curves at low-to-moderate noise suggests room for improvement over the linearized controller, which is only optimal in the linear-quadratic setting.

\section{Conclusion}
\label{sec:conclusion}
We derived a free energy lower bound on the minimum expected cost of
any causal feedback controller under partial observations and
introduced a self-consistent refinement that exploits the coupling
between achievable cost and sensor information. The resulting
fixed-point equation has a unique solution, computable by bisection,
that is provably tighter than the open-loop estimate. Under matched actuator noise, the free energy minimization is certifiably convex and numerically computable. On both linear and nonlinear problems, the self-consistent bound captures a substantial fraction of the
achievable cost. Together, these results provide a practical tool for certifying task feasibility and guiding sensor
selection for systems where domain-specific bounds do not apply.

Several extensions are natural. A rigorous finite sample analysis involving concentration inequalities is the primary next step, and has already been achieved for the TRIP bounds \cite{Majumdar22, Majumdar23}. Incorporating the Markov
structure of the dynamics into the MI
optimization~\eqref{eq:mi_opt}---e.g., via a covariance
propagation constraint---would tighten the CI certificate without
affecting the validity of the bound. The free energy formulation
also shares structural parallels with the Sagawa--Ueda generalization
of the second law~\cite{Sagawa10}, suggesting a deeper connection
between the thermodynamic cost of control and
information processing that we intend to develop in future work.

\bibliographystyle{IEEEtran}
\bibliography{IEEEabrv, bibliography}

\end{document}